\documentclass{tpms-lmyversion}

\usepackage{url}
\usepackage{enumerate}

\newtheorem{thm}{Theorem}[section]
\newtheorem{lem}[thm]{Lemma}
\newtheorem{prop}[thm]{Proposition}

\theoremstyle{definition}
\newtheorem{defn}[thm]{Definition}
\newtheorem{example}[thm]{Example}

\newtheorem{notn}[thm]{Notation}

\theoremstyle{remark}

\numberwithin{equation}{section}

\begin{document}

\title{Comparing the $G$-Normal Distribution to its Classical Counterpart}


\author{Erhan Bayraktar}
\address{Department of Mathematics, University of Michigan, Ann Arbor, Michigan 48109}
\email{erhan@umich.edu}
\thanks{This work is supported by the National Science Foundation under grant DMS-0955463. We gratefully \\ \indent acknowledge the anonymous referee for valuable advice on improving the manuscript.}

\author{Alexander Munk}
\address{Department of Mathematics, University of Michigan, Ann Arbor, Michigan 48109}
\email{amunk@umich.edu}






\begin{abstract}
In one dimension, the theory of the $G$-normal distribution is well-developed, and many results from the classical setting have a nonlinear counterpart. Significant challenges remain in multiple dimensions, and some of what has already been discovered is quite nonintuitive. By answering several classically-inspired questions concerning independence, covariance uncertainty, and behavior under certain linear operations, we continue to highlight the fascinating range of unexpected attributes of the multidimensional $G$-normal distribution.
\end{abstract}

\maketitle

\section{Introduction}

The $G$-framework, which includes the $G$-normal distribution and $G$-Brownian motion, was initially motivated by the study of risk measures and pricing under volatility uncertainty. Roughly speaking, one can think of these objects as the appropriate analogues of their classical namesakes in a setting of model uncertainty where the relevant collection of probability measures may be singular.

Activity in this area has been considerable since its introduction by Peng (see \cite{peng+abel}, \cite{peng+multidim}), and developments have proceeded at a rapid pace. A great variety of standard theorems from classical probability and stochastic analysis now have versions in the $G$-setting including the law of large numbers (see \cite{peng+lln+clt}, \cite{peng+clt}), the central limit theorem (see \cite{peng+lln+clt}, \cite{peng+clt}, \cite{li+shi}, \cite{hu+zhou+clt}, \cite{zhang+chen}), the martingale representation theorem (see \cite{soner+touzi+zhang}, \cite{song+mart}, \cite{peng+song+zhang}), L{\'e}vy's martingale characterization theorem (see \cite{xu+zhang}, \cite{xu+zhang+prop}, \cite{lin+rep}, \cite{lin}, \cite{song}), and Girsanov's theorem (see \cite{xu+shang+zhang}, \cite{osuka}, \cite{hu+ji+peng+song}). Substantial progress and extensions of this work have been completed in many other directions as well (e.g., see \cite{nutz}, \cite{dolinsky}, \cite{anton}, \cite{nutz+van+handel}, \cite{dolinsky+nutz+soner}). Readers interested in survey articles are referred to \cite{peng+ln}, \cite{peng+survey}, and \cite{peng+sc}.

Fundamental issues linger, especially in multiple dimensions. Much of $G$-stochastic analysis is built upon the $G$-normal distribution, and yet, many important elementary questions about this distribution remain unanswered. Some of what is known is also rather startling. For instance, the following result about the classical normal distribution is false in the $G$-setting (see \cite{peng+sc}):
\begin{quote}
For any $n$-dimensional random vector $Z$, if $\langle v , Z \rangle$ is a normal random variable for all $v \in \mathbb{R}^n$, then $Z$ is a normal random vector.
\end{quote}

Our intuition cannot be trusted when turned to $G$-normal  random vectors. Properties such as the one just mentioned fail due to the nonlinearity of the expectation operator in this framework. Other obstacles include the lack of well-understood tools from the classical theory (e.g., characteristic functions and density functions). Also, the distributional uncertainty associated to a $G$-normal random vector is far more complex than its initial appearance suggests, since viewing a $G$-normal random vector as having some fixed but unknown covariance matrix is usually incorrect.

Faced with these challenges, we asked to what extent additional properties of the classical normal distribution hold for its $G$-counterpart, particularly focusing on behaviors under various linear operations and the relationship between coordinate independence and the covariance matrix. We present our findings concerning the following classical theorems:
\begin{enumerate}[(i)]
\item Let $Z_1, \dots , Z_n$ be i.i.d. normal random variables. If  
\begin{equation*}
U = \displaystyle\sum_{i=1}^n a_i Z_i \hspace{5mm} \text{and} \hspace{5mm} V = \displaystyle\sum_{i=1}^n b_i Z_i
\end{equation*}
for real numbers $a_i, b_i$ satisfying 
\begin{equation*}
\displaystyle\sum_{i=1}^n a_i b_i = 0 ,
\end{equation*}
then $U$ and $V$ are independent.\\
\item Let $Z_1, \dots , Z_n$ be independent normal random variables. For any $m \times n$ real matrix $A$, if 
\begin{equation*}
Z = \left( Z_1 , \dots , Z_n \right)^\top , 
\end{equation*}
then $AZ$ is an $m$-dimensional normal random vector.\\
\item The covariance matrix of a normal random vector is diagonal if and only if its coordinates are (mutually) independent normal random variables.\\
\item If $Z$ is an $n$-dimensional normal random vector, then there exists an invertible $n \times n$ matrix $A$ such that the coordinates of $AZ$ are independent.\\
\end{enumerate}

Theorem \ref{new lin transf} reveals that (i) is no longer true in the $G$-setting. We show in Theorem \ref{my lin transf thm} that (ii) no longer holds either. While this was already known in a few special cases (e.g., see \cite{peng+sc}), our work explores a broad new class of examples and illuminates a surprising dichotomy depending on the rank of the matrix. Theorems \ref{diag not indep} and \ref{thm res 15} indicate that while (iii) is partially true for the $G$-normal distribution, unexpected new constraints on the coordinates are introduced. We end by demonstrating with Theorem \ref{invert not indep} that the analogue of (iv) is false. 

Our proofs often take advantage of a strange phenomenon in this setting: independence can be asymmetric, i.e., $Y$ can be independent from $X$ even if $X$ is not independent from $Y$. The general strategy is to show that this is incompatible with the symmetry relations imposed by the $G$-heat equation associated to the $G$-normal distribution, given a careful choice of parameters. 

These insights expand our knowledge of the remarkable series of behaviors exhibited by the multidimensional $G$-normal distribution. While the ultimate goal is to use these results to broaden our knowledge of $G$-stochastic analysis and its related financial applications, we believe that many more surprises lurk in the answers to further theoretical questions about this object. 

Readers unfamiliar with this area can find a short treatment of relevant background material in Section \ref{back sect}. The specific setup necessary for the statement of our results is in Section \ref{setup sect}. Our main results are contained in Section \ref{mr sect}. 

\section{Background}\label{back sect}

We begin with a brief survey of the theory of sublinear expectation spaces and the $G$-normal distribution. Our focus will be restricted to only those results that are directly needed for our work in the sequel. Readers interested in a more thorough treatment can find further details in \cite{peng+ln}, \cite{peng+survey}, \cite{peng+sc}, \cite{hu+li}, or \cite{luo+jia}, the references from which our discussion is adapted.

Throughout, we let $\Omega$ be a given set and $\mathcal{H}$ be a space of real-valued functions defined on $\Omega$. One should understand $\mathcal{H}$ as a space of random variables on $\Omega$. We will only place minimal emphasis on $\Omega$ and $\mathcal{H}$, but we suppose that $\mathcal{H}$ 
\begin{enumerate}[(i)]
\item is a linear space, \\
\item contains all constant functions, and \\
\item contains $\varphi \left( X_1 , X_2, \dots , X_n \right)$ for every $X_1 , X_2, \dots , X_n \in \mathcal{H}$ and $\varphi \in C_{l.Lip} \left( \mathbb{R}^n \right)$, where $C_{l.Lip} \left( \mathbb{R}^n \right)$ is the set of functions such that there exists $C >0$ and $m \in \mathbb{N}$ (depending on $\varphi$) satisfying
\begin{equation*}
\left| \varphi \left(x \right) - \varphi \left( y \right) \right| \leq C \left( 1 + \left|x \right|^m + \left|y\right|^m \right) \left|x -y \right|
\end{equation*}
for all $x$, $y \in \mathbb{R}^n$. \\
\end{enumerate} 
Our specific choice of test functions, $C_{l.Lip} \left( \mathbb{R}^n \right)$, is only a matter of convenience. Other spaces are also commonly used.

\begin{defn}
A {\it sublinear expectation} is a function $ \hat{\mathbb{E}} : \mathcal{H} \longrightarrow \mathbb{R}$ which is 
\begin{enumerate}[(i)]
\item Monotonic: $\hat{\mathbb{E}} \left[X \right] \leq \hat{\mathbb{E}} \left[Y \right]$ if $X \leq Y$, \\
\item Constant-preserving: $\hat{\mathbb{E}} \left[c \right] = c$ for any $c\in \mathbb{R}$, \\
\item Sub-additive: $\hat{\mathbb{E}} \left[ X + Y\right] \leq  \hat{\mathbb{E}} \left[X \right] + \hat{\mathbb{E}} \left[Y \right]$, and \\
\item Positive homogeneous: $\hat{\mathbb{E}} \left[ \lambda X\right] =  \lambda \hat{\mathbb{E}} \left[X \right]$ for $\lambda \geq 0$. \\
\end{enumerate}
The triple $\left( \Omega, \mathcal{H} , \hat{\mathbb{E}} \right)$ is called a {\it sublinear expectation space}.
\end{defn}

\begin{notn}
Unless stated otherwise, we will always work on some sublinear expectation space $\left( \Omega, \mathcal{H} , \hat{\mathbb{E}} \right)$, and when we have random variables $X_1 , X_2, \dots , X_n \in \mathcal{H}$, we will say that $X$ is an $n$-dimensional random vector and write $X \in \mathcal{H}^n$.
\end{notn}

Great caution is required when manipulating expressions with sublinear expectations due to (iii) and (iv). Most familar operations from the classical theory are simply no longer valid. One situation where a standard technique can be applied is the following.

\begin{lem}\label{lem sub exp}
Consider two random variables $X, Y \in \mathcal{H}$ such that $\hat{\mathbb{E}} \left[Y \right] = - \hat{\mathbb{E}} \left[-Y \right]$. Then 
\begin{equation*}
\hat{\mathbb{E}} \left[X + \alpha Y \right] = \hat{\mathbb{E}} \left[X \right] + \alpha \hat{\mathbb{E}} \left[Y \right] 
\end{equation*}
for all $\alpha \in \mathbb{R}$. 
\end{lem}

In the literature, random variables such as $Y$ above are said to have {\it no mean-uncertainty}, a notion which also arises in the context of {\it symmetric $G$-martingales}. We will resort to a notable consequence of this result again and again: if $\hat{\mathbb{E}} \left[Y \right] = \hat{\mathbb{E}} \left[-Y \right] = 0$, then for all $\alpha \in \mathbb{R}$, 
\begin{equation}\label{lin mean cert}
\hat{\mathbb{E}} \left[X + \alpha Y \right]= \hat{\mathbb{E}} \left[X \right].
\end{equation}
\begin{defn}
An $n$-dimensional random vector $Y \in \mathcal{H}^n$ is said to be {\it independent} from an $m$-dimensional random vector $X \in \mathcal{H}^m$ if for all $\varphi \in C_{l.Lip} \left(\mathbb{R}^{m+n} \right)$, we have 
\begin{equation*}
\hat{\mathbb{E}} \left[ \varphi \left( X , Y\right) \right] = \hat{\mathbb{E}} \left[  \hat{\mathbb{E}} \left[ \varphi \left( x , Y\right) \right]_{ x = X} \right] .
\end{equation*}
\end{defn}

As we mentioned previously, independence can be asymmetric when $\hat{\mathbb{E}}$ is not a linear expectation. A now standard example which we will refer to later illustrates this concretely. 

\begin{example}\label{std examp}
Consider random variables $X, Y \in \mathcal{H}$ such that 
\begin{enumerate}[(i)]
\item $\hat{\mathbb{E}} \left[ X \right] = - \hat{\mathbb{E}} \left[- X \right] = 0$,\\
\item $\hat{\mathbb{E}} \left[\lvert X \rvert \right] > 0$, and \\
\item $\hat{\mathbb{E}} \left[ Y^2 \right] > -\hat{\mathbb{E}} \left[ -Y^2 \right]$.\\
\end{enumerate}
If $X$ is independent from $Y$, 
\begin{equation*}
\hat{\mathbb{E}} \left[ X Y^2 \right] = 0 ,
\end{equation*}
while if $Y$ is independent from $X$, 
\begin{equation*}
\hat{\mathbb{E}} \left[ X Y^2 \right] > 0  . 
\end{equation*}
A broad class of situations where (i) , (ii), and (iii) are satisfied occurs when $X \sim \mathcal{N} \left( 0 , \left[ \underline{\sigma}_1^{2} , \overline{\sigma}_1^{2} \right] \right)$ for $0 < \overline{\sigma}_1^{2}$ and $Y \sim \mathcal{N} \left( 0 , \left[ \underline{\sigma}_2^{2} , \overline{\sigma}_2^{2} \right] \right)$ for $\underline{\sigma}_2^{2} < \overline{\sigma}_2^{2}$ (see below for notation).
\end{example}

Ignoring trivial cases, one can actually characterize the distribution of $X$ and $Y$ if $X$ is independent from $Y$ and vice versa (see Proposition \ref{hu+li thm} below). Still, observe that if $\hat{\mathbb{E}}$ is a linear expectation, this definition is equivalent to the classical one. 

\begin{defn}
Let $X$ be an $n$-dimensional random vector, i.e., $X \in \mathcal{H}^n$.
\begin{enumerate}[(i)]
\item The {\it distribution} of $X$, $\mathbb{F}_{X}$, is defined on $C_{l.Lip} \left( \mathbb{R}^n \right)$  by 
\begin{equation*}
\mathbb{F}_{X} \left( \varphi \right) = \hat{\mathbb{E}} \left[ \varphi \left( X \right) \right] .
\end{equation*}
\item $X$ has {\it distributional uncertainty} if $\mathbb{F}_X$ is not a linear expectation ( $\left( \mathbb{R}^n ,  C_{l.lip} \left( \mathbb{R}^n \right), \mathbb{F}_X \right)$ is always a sublinear expectation space). \\
\item If $Y \in \mathcal{H}^n$ is another $n$-dimensional random vector, then $X$ and $Y$ are {\it identically distributed}, denoted $X \sim Y$, if 
\begin{equation*}
\mathbb{F}_X = \mathbb{F}_Y  .
\end{equation*} 
\item If $X$ and $Y$ are identically distributed and $Y$ is independent from $X$, then $Y$ is an {\it independent copy} of $X$. \\
\end{enumerate}
\end{defn}

This notion of being identically distributed is also equivalent to the classical defintion if $\hat{\mathbb{E}}$ is a linear expectation. The sublinear case possesses many interesting new features, but we will only need to know about those pertaining to the {\it $G$-normal distribution} and the {\it maximal distribution}. 

\begin{defn}\label{def g norm}
An $n$-dimensional random vector $X \in \mathcal{H}^n$ is said to be {\it $G$-normally distributed} if for any independent copy of $X$, say $\bar{X}$, we have 
\begin{equation*}
a X + b \bar{X} \sim \sqrt{a^2 + b^2} X 
\end{equation*}
for all $a , b \geq 0$. 
\end{defn}

``$G$'' refers to the sublinear function defined on the space of $n \times n$ symmetric matrices,  $\mathbb{S} \left(n \right)$, by 
\begin{equation}\label{G eqn 1}
G \left( A \right) = \frac{1}{2} \hat{\mathbb{E}} \left[ \langle A X , X \rangle \right] .
\end{equation}
For each such function, there exists a unique bounded, closed, convex subset $\Gamma$ of the $n \times n$ positive semidefinite matrices, $\mathbb{S}^+ \left(n \right)$, such that 
\begin{equation}\label{G eqn 2}
G \left(A \right) = \frac{1}{2} \displaystyle\sup_{B \in \Gamma} \text{tr} \left[AB \right] .
\end{equation}
Conversely, given any $\Gamma$ with these properties, there exists a $G$-normal random vector $X$ such that (\ref{G eqn 1}) and (\ref{G eqn 2}) hold. 

$\Gamma$ completely determines the distribution of a $G$-normal random vector $X$, and in fact, one can loosely interpret $\Gamma$ as describing the covariance uncertainty of $X$. As we remarked above, some care must be exercised since viewing $X$ as possessing a classical normal distribution with some fixed but unknown covariance matrix selected from $\Gamma$ is not generally correct. However, if $\Gamma$ contains only one element, then $X$ is a classical normal random vector with mean zero and covariance $\Gamma$.

\begin{notn}
We write $X \sim \mathcal{N} \left( 0 , \Gamma \right)$. If $n = 1$, then $\Gamma =  \left[ \underline{\sigma}^{2} , \overline{\sigma}^{2} \right]$ for $0 \leq \underline{\sigma}^{2} \leq \overline{\sigma}^{2}$ given by
\begin{equation*}
\overline{\sigma}^2 = \hat{\mathbb{E}} \left[ X^2 \right] \quad \text{and} \quad \underline{\sigma}^2 = - \hat{\mathbb{E}} \left[ -X^2 \right]  .
\end{equation*}
\end{notn}

We will frequently use the following important basic properties of a $G$-normal random vector. 

\begin{lem}\label{basic props known}
Let $X = \left( X_1 , \dots , X_n \right)$ be an $n$-dimensional $G$-normal random vector, i.e., $X \sim \mathcal{N} \left( 0  , \Gamma \right)$. Then
\begin{enumerate}[(i)]
\item $\hat{\mathbb{E}} \left[X_i \right] = \hat{\mathbb{E}} \left[-X_i \right] = 0$  for all $i$;\\
\item $-X$ and $X$ are identically distributed, i.e., $-X \sim \mathcal{N} \left( 0  , \Gamma \right)$;\\
\item\label{peng inn prod} for all $v \in \mathbb{R}^n$, the random variable $\langle v, X \rangle$ is $\mathcal{N} \left( 0, \left[ \underline{\sigma}^{2} , \overline{\sigma}^{2} \right] \right)$-distributed, where 
\begin{equation*}
\overline{\sigma}^{2} = \hat{\mathbb{E}} \left[\langle  v, X \rangle^2  \right] \quad \text{and} \quad \underline{\sigma}^{2} = - \hat{\mathbb{E}} \left[-\langle v, X \rangle^2  \right] \text{; and}
\end{equation*}
\item for all $m \times n$ real matrices $M$, $MX \sim \mathcal{N} \left( 0 , M \Gamma M^\top \right)$. \\
\end{enumerate}
\end{lem}

Perhaps the deepest known property of a $G$-normal random vector is its intimate connection to the so-called ``$G$-heat equation'', a parabolic PDE generalizing the classical heat equation. 

\begin{prop}\label{g heat eqn thm}
Let $X$ be an $n$-dimensional $G$-normal random vector, i.e., $X \sim \mathcal{N} \left( 0  , \Gamma \right)$. For all $\varphi \in C_{l.Lip} \left(\mathbb{R}^n \right)$, the function
\begin{equation*}
u \left(t, x \right) = \hat{\mathbb{E}} \left[ \varphi \left( x + \sqrt{t}X \right) \right]  , \quad \left(t, x\right) \in \left[0 , \infty \right) \times \mathbb{R}^n
\end{equation*}
is the unique viscosity solution of the following PDE:
\begin{align*}
\partial_t u - G \left( D^2 u \right) &= 0 , \quad \left( t , x \right) \in \left( 0 , \infty \right) \times \mathbb{R}^n \\
u \left( 0 , x \right)  &= \varphi \left( x\right)  , \quad x \in \mathbb{R}^n  , 
\end{align*}
where $D^2 u = \left( \partial^{2}_{x^i x^j } u \right)_{ij}$. 
\end{prop}

We will not need the full theory of viscosity solutions in our development, as the solution to the equation above is actually a classical solution if $G$ is non-degenerate. 

We will use the remainder of this section only for the proof of Theorem \ref{thm res 15}.

\begin{notn}
We let $C_{b.Lip} \left( \mathbb{R}^n \right)$ denote the space of bounded Lipschitz functions on $\mathbb{R}^n$. 
\end{notn}

We recall that \cite{hu+li} uses $C_{b.Lip} \left( \mathbb{R}^n \right)$ as the space of test functions instead of $C_{l.Lip} \left( \mathbb{R}^n \right)$. This technicality does not matter for our proof of Theorem \ref{thm res 15}, as we will explain later.
\begin{defn}\label{max def}
An $n$-dimensional random vector $X \in \mathcal{H}^n$ is called {\it maximally distributed} if there exists a closed set $\Gamma \subset \mathbb{R}^n$ such that 
\begin{equation*}
\hat{\mathbb{E}} \left[ \varphi \left( X \right) \right] = \displaystyle\sup_{x \in \Gamma} \varphi \left( x \right) 
\end{equation*}
for all $\varphi \in C_{b.Lip} \left( \mathbb{R}^n \right)$.
\end{defn}

One can understand random variables of this kind as analogues of constants in the sublinear setting. In particular, if $X$ is a $G$-normal random variable with $X \sim \mathcal{N} \left( 0, \left[ \underline{\sigma}^{2} , \overline{\sigma}^{2} \right] \right)$ for some $\underline{\sigma}^2 < \overline{\sigma}^2$, then $X$ is not maximally distributed. Observe that we have presented the definition from \cite{hu+li} rather than the original version in \cite{peng+sc}, as we will need the following theorem from \cite{hu+li}. 

\begin{prop}\label{hu+li thm}
Suppose that the random variable $W \in \mathcal{H}$ has distributional uncertainty and that the random variable $W^\prime \in \mathcal{H}$ is not a constant. If $W$ is independent from $W^\prime$ and vice versa, then $W$ and $W^\prime$ must be maximally distributed.
\end{prop}

\section{Basic Setup}\label{setup sect}

Throughout our work below, we consider random vectors $X, Y \in \mathcal{H}^n$ such that
\begin{enumerate}[(i)]
\item $X = \left(X_1 , \dots , X_n \right)$ is a $G$-normal random vector, i.e., $X \sim \mathcal{N} \left( 0, \Gamma \right)$; and \\
\item $Y =  \left( Y_1 , \dots , Y_{n} \right)$, where 
\begin{enumerate}
\item\label{sigma ineq} $Y_1 \sim \mathcal{N} \left( 0, \left[ \underline{\sigma}^{2} , \overline{\sigma}^{2} \right] \right)$ for some $0 < \underline{\sigma}^2 < \overline{\sigma}^2$; \\
\item $Y_{i+1}$ and $Y_i$ are identically distributed, i.e., $Y_{i+1} \sim Y_{i}$; and\\
\item $Y_{i+1}$ is independent from $\left( Y_1 , \dots , Y_{i} \right)$.\\
\end{enumerate}
\end{enumerate}
Although a seemingly innocent construction, we will see that $Y$ is a rich source of counterexamples when evaluating whether or not standard classical theorems about the normal distribution hold in the $G$-framework.
The inequality
\begin{equation*}
\underline{\sigma}^2 < \overline{\sigma}^2
\end{equation*}
is critical for us, as it implies that the coordinates of $Y$ are not classical normal random variables. Because our objective is to compare the properties of the classical normal distribution with its $G$-counterpart, we have added the assumption
\begin{equation*}
0 < \underline{\sigma}^2
\end{equation*}
for convenience. As observed above, this condition ensures that the solution to the $G$-heat equation is actually classical. 

Whenever restating classical results to facilitate our comparisons, we will always call classical random vectors ``$Z$''.

\section{Main Results}\label{mr sect}

\subsection{Behavior Under Linear Combinations}\label{lin comb sect}

Recall that in the classical setting, we have the following result:
\begin{quote} 
Let $Z_1, \dots , Z_n$ be i.i.d. normal random variables. If  
\begin{equation*}
U = \displaystyle\sum_{i=1}^n a_i Z_i \quad \text{and} \quad V = \displaystyle\sum_{i=1}^n b_i Z_i
\end{equation*}
for real numbers $a_i$, $b_i$ satisfying 
\begin{equation*}
\displaystyle\sum_{i=1}^n a_i b_i = 0 ,
\end{equation*}
then $U$ and $V$ are independent.
\end{quote}
The corresponding statement does not hold in the $G$-framework.

\begin{thm}\label{new lin transf}
Let $a_1 = 1$, $a_2 = 1$, $b_1 = 1$, and $b_2 = -1$. Set the remaining constants equal to zero, i.e.,
\begin{equation*}
 U = Y_1 + Y_2 \hspace{5mm} \text{and} \hspace{5mm} V = Y_1 -Y_2 .
\end{equation*} Then 
\begin{equation*}
\displaystyle\sum_{i=1}^n a_i b_i = 0,
\end{equation*}
but $U$ is not independent from $V$ and vice versa.
\end{thm}

In fact, it is not yet known if any non-trivial linear combination of this kind will produce independent random variables. An important classical characterization of the normal distribution, the Skitovich-Darmois theorem, is related to the independence of such linear combinations, and whether or not a $G$-version of this holds is also unknown. 

\begin{proof}
Our strategy will be to show that the random vectors $\left( U, V \right)$ and $\left(V , U \right)$ are identically distributed. On the other hand, if $U$ were independent from $V$ or vice versa, the resulting destruction of symmetry would make this impossible.

A simple calculation shows that $-Y_2$ is an independent copy of $Y_1$, which means that 
\begin{equation*}
V \sim \sqrt{2} Y_1 \sim \mathcal{N} \left( 0, \left[ 2 \underline{\sigma}^{2} , 2 \overline{\sigma}^{2} \right] \right) 
\end{equation*}
by Definition \ref{def g norm}. The same is true for $U$. 

By Example \ref{std examp}, if $U$ is independent from $V$, then
\begin{equation*}
\hat{\mathbb{E}} \left[  U V^2 \right] = 0 
\end{equation*}
and
\begin{equation*}
\hat{\mathbb{E}} \left[ V U^2 \right]  > 0 .
\end{equation*}
If $V$ is independent from $U$, then 
\begin{equation*}
\hat{\mathbb{E}} \left[ V U^2 \right] = 0 
\end{equation*}
and
\begin{equation*}
\hat{\mathbb{E}} \left[  U V^2 \right] > 0  .
\end{equation*}
Hence, exactly one of these must hold:
\begin{enumerate}[(i)]
\item $U$ is independent from $V$ but $V$ is not independent from $U$.\\
\item $V$ is independent from $U$ but $U$ is not independent from $V$.\\
\item $U$ is not independent from $V$ and $V$ is not independent from $U$.\\
\end{enumerate}

Let $S : \mathbb{R}^2 \rightarrow \mathbb{R}^2$ be defined by 
\begin{equation*}
S \left( x, y \right) = \left( x - y , x+y \right) \quad \text{for all } \left(x, y \right) \in \mathbb{R}^2 .
\end{equation*}
Observe that $\varphi \circ S \in C_{l.Lip} \left( \mathbb{R}^2 \right)$ for any $\varphi \in C_{l.Lip} \left( \mathbb{R}^2 \right)$. Since $Y_2$ and $-Y_2$ are each independent copies of $Y_1$, for any $\varphi \in C_{l.Lip} \left( \mathbb{R}^2 \right)$,
\begin{align*}
\hat{\mathbb{E}} \left[ \varphi \left( V, U \right) \right] &=\hat{\mathbb{E}} \left[ \varphi \left( Y_1 -Y_2, Y_1 + Y_2 \right) \right] \\
&=  \hat{\mathbb{E}} \left[\left( \varphi \circ S \right) \left( Y_1, Y_2 \right) \right] \\
&= \hat{\mathbb{E}} \left[ \hat{\mathbb{E}} \left[ \left( \varphi \circ S \right) \left( \bar{x}, Y_2 \right) \right]_{\bar{x}=Y_1 } \right] \\
&= \hat{\mathbb{E}} \left[ \hat{\mathbb{E}} \left[ \left( \varphi \circ S \right) \left( \bar{x}, -Y_2 \right) \right]_{\bar{x}=Y_1 } \right] \\
&=  \hat{\mathbb{E}} \left[\left( \varphi \circ S \right) \left( Y_1, -Y_2 \right) \right] \\
&= \hat{\mathbb{E}} \left[ \varphi \left( Y_1 + Y_2, Y_1 - Y_2 \right) \right] \\
&= \hat{\mathbb{E}} \left[ \varphi \left( U, V \right) \right] .
\end{align*}

Applying this equality to $\varphi \left( x , y \right) = x y^2$, we have
\begin{equation*}
\hat{\mathbb{E}} \left[ V U^2 \right] = \hat{\mathbb{E}} \left[ U V^2  \right] ,
\end{equation*}
which implies that $U$ is not independent from $V$ and vice versa.
\end{proof}

\subsection{Behavior Under General Linear Transformations}\label{lin tranf sect}

Allowing the degenerate case, we have another important property of the classical normal distribution:
\begin{quote}
Let $Z_1, \dots , Z_n$ be independent normal random variables. For any $m \times n$ real matrix $A$, if 
\begin{equation*}
Z = \left( Z_1 , \dots , Z_n \right)^\top , 
\end{equation*}
then $AZ$ is an $m$-dimensional normal random vector.
\end{quote}
The situation in the $G$-framework is far more delicate.
\begin{thm}\label{my lin transf thm}
For any $m \times n$ real matrix $A$,
\begin{enumerate}[(i)]
\item $\left\langle v , AY \right\rangle$  is a $G$-normal random variable with 
\begin{equation*}
\left\langle v , AY \right\rangle \sim \mathcal{N} \left( 0, \left[ \left\| v^\top A\right\|^2 \underline{\sigma}^{2} ,  \left\| v^\top A\right\|^2 \overline{\sigma}^{2} \right] \right)
\end{equation*}
for all $v \in \mathbb{R}^m$;\\
\item if $A$ has rank less than or equal to one, $AY$ is an $m$-dimensional $G$-normal random vector, or more precisely, $AY \sim \mathcal{N} \left( 0 , \Gamma^\prime \right)$, where
\begin{equation*}
\Gamma^\prime = \left\{ u r u^\top : r \in \left[ \left\| w \right\|^2 \underline{\sigma}^{2}  , \left\| w \right\|^2 \overline{\sigma}^{2} \right] \right\}
\end{equation*}
and $A = u w^{\top}$  for $u \in \mathbb{R}^m$, $w \in \mathbb{R}^n$; and\\
\item if $A$ is invertible, $AY$ is not $G$-normally distributed. (In particular, $Y$ is not $G$-normally distributed.)\\
\end{enumerate}
\end{thm}

One would expect (i) and (ii); however, (iii) is quite surprising both because of the bifurcation it reveals and its relation to classical theorems. While it was previously known that the classical property above failed to be true in the $G$-setting (e.g., see \cite{peng+sc}), our result provides an expansive new series of cases illustrating this failure. It serves the same purpose with respect to the following classical statement as well:
\begin{quote}
For any $n$-dimensional random vector $Z$, if $\langle v , Z \rangle$ is a normal random variable for all $v \in \mathbb{R}^n$, then $Z$ is an $n$-dimensional normal random vector.
\end{quote}
Whether or not $AY$ is a $G$-normal random vector if $A$ is non-invertible but has rank strictly greater than one remains unclear. 

The most difficult part of the proof is the following lemma, which is a small extension of Exercise 1.15 in \cite{peng+sc}. This lemma will be critical for our proof of Theorems \ref{diag not indep}, \ref{thm res 15}, and \ref{invert not indep} as well.

\begin{lem}\label{thm res 3}
\noindent Let $\alpha > 0$. Suppose that $W_1$ and $W_2$ are two $G$-normal random variables such that
\begin{enumerate}[(i)]
\item $W_1 \sim \mathcal{N} \left( 0, \left[ \underline{\sigma}^2 , \overline{\sigma}^2 \right] \right)$; and \\
\item $W_2 \sim \mathcal{N} \left( 0, \left[ \alpha \underline{\sigma}^2 , \alpha \overline{\sigma}^2 \right] \right)$. \\
\end{enumerate}
If either $W_2$ is independent from $W_1$ or vice versa, then $W =\left( W_1 , W_2 \right)^\top$ is not a 2-dimensional $G$-normal random vector.
\end{lem}

\begin{proof}
We first consider the case where $W_2$ is independent from $W_1$. Suppose instead that $W$ is a $2$-dimensional $G$-normal random vector. Our initial step will be to compute the corresponding $G$-heat equation, which we will use to establish an identity relating the distributions of the random vectors $W =\left( W_1 , W_2 \right)^\top$ and $\left( W_2 , W_1 \right)^\top$. The conclusion will be reached by showing that this ``symmetry'' contradicts the asymmetry induced by our independence assumption.

We can find a bounded, closed, convex subset $\Gamma \subset \mathbb{S}^+ \left(2\right)$ such that 
\begin{equation*}
\frac{1}{2} \hat{\mathbb{E}} \left[ \langle AW , W \rangle \right] = G \left( A \right) = \frac{1}{2} \displaystyle\sup_{B \in \Gamma} \text{tr} \left[ AB \right] 
\end{equation*}
for all $A \in \mathbb{S} \left(2 \right)$. 
Now for all 
\begin{equation*} A = \begin{bmatrix}
a_{11} & a_{12} \\
a_{12} & a_{22} \\
\end{bmatrix} \in \mathbb{S} \left(2 \right)  ,
\end{equation*}
we have 
\begin{align*}
G \left( A \right) &=  \frac{1}{2} \hat{\mathbb{E}} \left[ \left\langle\begin{bmatrix}
a_{11} & a_{12} \\
a_{12} & a_{22} \\
\end{bmatrix} \begin{bmatrix}
W_1 \\
W_2 \\
\end{bmatrix} , \begin{bmatrix}
W_1 \\
W_2 \\ 
\end{bmatrix} \right\rangle \right] \\
&=  \frac{1}{2} \hat{\mathbb{E}} \left[ a_{11} W_{1}^2 + 2 a_{12} W_1 W_2 + a_{22} W_{2}^2 \right] \\
& = \frac{1}{2} \hat{\mathbb{E}} \left[ \hat{\mathbb{E}} \left[ a_{11} \bar{x}^2 + 2 a_{12} \bar{x} W_2 + a_{22} W_{2}^2 \right]_{\bar{x}=W_1 } \right] \\
& = \frac{1}{2} \hat{\mathbb{E}} \left[ \left( a_{11} \bar{x}^2  + 2 a_{12} \bar{x} \hat{\mathbb{E}} \left[  W_2 \right] + \hat{\mathbb{E}} \left[  a_{22} W_{2}^2 \right] \right)_{ \bar{x}=W_1} \right] \\
& = \frac{1}{2} \hat{\mathbb{E}} \left[ \left( a_{11} \bar{x}^2 + \alpha \overline{\sigma}^2 \left(a_{22}^+ \right) - \alpha \underline{\sigma}^2 \left(a_{22}^- \right)  \right)_{ \bar{x}=W_1 } \right] \\
& = \frac{1}{2} \hat{\mathbb{E}} \left[ a_{11} W_{1}^2 \right] +  \frac{\alpha}{2} \left(\overline{\sigma}^2 \left(a_{22}^+ \right) - \underline{\sigma}^2 \left(a_{22}^- \right) \right) \\
&= \frac{1}{2} \left( \overline{\sigma}^2 \left(a_{11}^+ \right) - \underline{\sigma}^2 \left(a_{11}^- \right) \right) +  \frac{\alpha}{2} \left(\overline{\sigma}^2 \left(a_{22}^+ \right) - \underline{\sigma}^2 \left(a_{22}^- \right) \right) \\
&= \bar{G} \left( a_{11} \right) + \alpha \bar{G} \left(a_{22} \right) ,
\end{align*}
where $\bar{G}$ is defined by 
\begin{equation*}
\bar{G} \left( x \right) = \frac{1}{2} \left( \overline{\sigma}^2 \left(x^+ \right) - \underline{\sigma}^2 \left(x^- \right) \right)
\end{equation*}
for all $x \in \mathbb{R}$. A quick calculation verifies that $\Gamma$ must then be given by
\begin{equation*}
\Gamma = \left\{ \begin{bmatrix}
r_1 & 0\\
0 & r_2 \\
\end{bmatrix} : r_1 \in \left[ \underline{\sigma}^2 , \overline{\sigma}^2 \right] , r_2 \in \left[ \alpha \underline{\sigma}^2 , \alpha \overline{\sigma}^2 \right] \right\} .
\end{equation*}

Let $\varphi \in C_{l.Lip} \left( \mathbb{R}^2 \right)$. Define the function $u$ by 
\begin{equation*}
u \left( t, x , y \right) = \hat{\mathbb{E}} \left[ \varphi \left( \left(x, y\right) + \sqrt{t} \left( W_1 , W_2 \right) \right) \right]
\end{equation*}
for all $\left( t, x , y\right) \in \left[ \left. 0 , \infty \right. \right) \times \mathbb{R}^2$. By Proposition \ref{g heat eqn thm}, $u$ is the unique viscosity solution to 
\begin{align*}
\partial_t u - \bar{G} \left( \partial^{2}_{xx} u \right) - \alpha \bar{G} \left( \partial_{yy}^2 u \right) &= 0 , \quad \left( t , x ,y \right) \in \left( 0 , \infty \right) \times \mathbb{R}^2 \\
u \left( 0 , x , y \right)  &= \varphi \left( x, y\right) , \quad \left(x , y \right) \in \mathbb{R}^2 .
\end{align*}
In fact, since  $0 < \underline{\sigma}^2$, $u$ is a classical solution. 

Let the functions $S$ and $\tilde{S}$ be given by 
\begin{equation*}
S \left( x , y \right) = \left( \frac{1}{\sqrt{\alpha}} y ,  \sqrt{\alpha} x \right) 
\end{equation*}
for all $\left(x, y \right) \in \mathbb{R}^2$ and
\begin{equation*}
\tilde{S} \left( t, x, y\right) = \left( t, \frac{1}{\sqrt{\alpha}} y ,  \sqrt{\alpha}x \right) 
\end{equation*}
for all $\left(t, x, y \right) \in \left[  \left. 0 , \infty \right. \right) \times \mathbb{R}^2$. 
Define a function $v$ by 
\begin{equation*}
v = u \circ \tilde{S} .
\end{equation*}
Then $v$ is a classical solution of 
\begin{align*}
\partial_t v - \bar{G} \left( \partial^{2}_{xx} v \right) - \alpha \bar{G} \left( \partial_{yy}^2 v \right) &= 0 , \quad \left( t , x ,y \right) \in \left( 0 , \infty \right) \times \mathbb{R}^2 \\
v \left( 0 , x, y \right) &=  \left( \varphi \circ S\right) \left( x, y\right) , \quad \left(x , y \right) \in \mathbb{R}^2 .
\end{align*}
Hence, 
\begin{equation*}
v \left( t, x, y \right) = \hat{\mathbb{E}} \left[ \left( \varphi \circ S  \right) \left( \left( x, y \right) + \sqrt{t} \left( W_1 , W_2 \right) \right) \right]  .
\end{equation*}
In particular, 
\begin{align*}
&\hat{\mathbb{E}} \left[ \varphi \left( \frac{1}{\sqrt{\alpha}}  W_2 , \sqrt{\alpha}  W_1 \right) \right] \\
&= \hat{\mathbb{E}} \left[ \left( \varphi \circ S  \right) \left( \left( 0, 0 \right) + \sqrt{1} \left( W_1 , W_2 \right) \right) \right] \\
&= v \left( 1, 0, 0 \right) \\
&= \left( u \circ \tilde{S} \right) \left(1, 0, 0\right) \\
&= u \left( 1, 0, 0 \right) \\
&= \hat{\mathbb{E}} \left[ \varphi \left( \left( 0, 0 \right) + \sqrt{1} \left( W_1 , W_2 \right) \right) \right] \\
&= \hat{\mathbb{E}} \left[ \varphi \left(  W_1 ,  W_2 \right) \right] . 
\end{align*}
Since $\varphi \in C_{l.Lip} \left( \mathbb{R}^2 \right)$ was arbitrary, applying this to the function $\varphi \left (x , y \right) = x y^2$ gives
\begin{equation*}
\sqrt{\alpha} \hat{\mathbb{E}} \left[ W_2 W_{1}^2 \right] = \hat{\mathbb{E}} \left[ W_1 W_{2}^2 \right] .
\end{equation*}

On the other hand, Example \ref{std examp} implies 
\begin{equation*}
\hat{\mathbb{E}} \left[ W_2 W_{1}^2 \right] = 0 
\end{equation*}
and 
\begin{equation*}
\hat{\mathbb{E}} \left[ W_1 W_{2}^2 \right] >0 , 
\end{equation*}
a contradiction. It follows that $W$ is not a $2$-dimensional $G$-normal random vector in this case.

To finish the proof, assume that $W_1$ is independent from $W_2$. If $W=\left( W_1 , W_2 \right)^\top$ were 2-dimensional $G$-normal random vector, then
\begin{equation*}
 \begin{bmatrix}
0 & 1\\
1 & 0\\
\end{bmatrix} \cdot  \begin{bmatrix}
W_1 \\
W_2 \\
\end{bmatrix} = \begin{bmatrix}
W_2 \\
W_1 \\
\end{bmatrix}
\end{equation*}
would be a $2$-dimensional $G$-normal random vector by Lemma \ref{basic props known}. This is impossible by what we just considered, so the result holds.
\end{proof}

The proof of the theorem is now straightforward.

\begin{proof}
Let $A = \left( a_{ij} \right)$. To prove (i), suppose $v = \left( v_1 , \dots ,  v_m \right)^\top \in \mathbb{R}^m$. For all $\varphi \in C_{l.Lip} \left( \mathbb{R} \right)$, 
\begin{align*}
\hat{\mathbb{E}} \left[ \varphi \left( \langle v , AY \rangle \right) \right] &= \hat{\mathbb{E}} \left[ \varphi \left( \left( v_1 \displaystyle\sum_{k = 1}^n a_{1k} Y_k \right) + \cdots + \left( v_m \displaystyle\sum_{k = 1}^n a_{mk} Y_k \right) \right) \right] \\
&= \hat{\mathbb{E}} \left[ \varphi \left( \left( \displaystyle\sum_{k = 1}^m  v_k a_{k1}  \right) Y_1 + \cdots + \left( \displaystyle\sum_{k = 1}^m v_k a_{kn}  \right) Y_n \right) \right] \\
&= \hat{\mathbb{E}} \left[ \varphi \left( \sqrt{\left( \displaystyle\sum_{k = 1}^m v_k a_{k1}  \right)^2 + \cdots + \left( \displaystyle\sum_{k = 1}^m v_k a_{kn}  \right)^2 }  Y_1 \right) \right] \\
&= \hat{\mathbb{E}} \left[ \varphi \left( \left\| v^\top A \right\| Y_1 \right) \right] ,
\end{align*}
which directly follows from Definition \ref{def g norm}.
By Lemma \ref{basic props known},
\begin{equation*}
\langle v , AY \rangle \sim \left\| v^\top A \right\| Y_1 \sim \mathcal{N} \left( 0, \left[ \left\| v^\top A \right\|^2 \underline{\sigma}^{2} , \left\| v^\top A\right\|^2 \overline{\sigma}^{2} \right] \right) .
\end{equation*}

For (ii), since $A$ has rank less than or equal to one, we can find $u = \left( u_1 , \dots , u_m \right)^\top \in \mathbb{R}^m$ and $w = \left( w_1 , \dots , w_n \right)^\top \in \mathbb{R}^n$ such that 
\begin{equation*}
A = u  w^\top .
\end{equation*}
By Lemma \ref{basic props known},
\begin{equation*}
w^\top Y  \sim \left( \sqrt{\displaystyle\sum_{i = 1}^n w_{i}^2 } \right) Y_1 \sim \mathcal{N} \left( 0, \left[ \left\| w \right\|^2 \underline{\sigma}^{2}  , \left\| w \right\|^2 \overline{\sigma}^{2} \right] \right)  .
\end{equation*}
Since $AY$ is given by 
\begin{equation*}
AY = u  \left( w^{\top}  Y \right) ,
\end{equation*}
another application of Lemma \ref{basic props known} implies the result.

To see that (iii) holds, let $B$ be the $2 \times n$ matrix 
\begin{equation*}
B = \begin{bmatrix}
1 & 0 & \cdots & \cdots & 0 \\
0 & 1 & 0 & \cdots & 0 \\
\end{bmatrix} .
\end{equation*}
Since
\begin{equation*}
\left( B A^{-1} \right) AY = BY = \begin{bmatrix}
Y_1 \\
Y_2 \\
\end{bmatrix} ,
\end{equation*}
Lemmas \ref{basic props known} and \ref{thm res 3} show that $AY$ cannot be an $n$-dimensional $G$-normal random vector.
\end{proof}

\subsection{Connections Between Covariance Uncertainty and Independence}\label{cov indep sect}

Classically, there is a tight relationship between the covariance matrix of a normal random vector and the independence of its coordinates:
\begin{quote}
The covariance matrix of a normal random vector is diagonal if and only if its coordinates are (mutually) independent normal random variables. 
\end{quote}
Once again, the analogous situation is more subtle in the $G$-setting. For instance, the forward direction is false.

\begin{thm}\label{diag not indep}
Let 
\begin{equation*}
\Gamma = \left\{ \begin{bmatrix}
r_1 & 0\\
0 & r_2 \\
\end{bmatrix} : r_1 , r_2 \in \left[ \underline{\sigma}^2 , \overline{\sigma}^2 \right] \right\} ,
\end{equation*}
so $X = \left( X_1 , X_2 \right)^\top$ and $X \sim \mathcal{N} \left(0 , \Gamma \right)$. $X_1$ is not independent from $X_2$ and vice versa.
\end{thm}

\begin{proof}
We will proceed by computing the distributions of the random variables $X_1$ and $X_2$. Then we will invoke Lemma \ref{thm res 3}, the impetus for our choice of this specific $\Gamma$.

Recall that for $A \in \mathbb{S} \left( 2 \right)$, 
\begin{equation*}
\frac{1}{2} \hat{\mathbb{E}} \left[ \langle AX, X \rangle \right] = G \left(A \right)  = \frac{1}{2} \displaystyle\sup_{B \in \Gamma} \text{tr} \left[ A B \right]  .
\end{equation*}
 In particular, 
\begin{align*}
\frac{1}{2} \hat{\mathbb{E}} \left[ X_1^2 \right] &=G \left( \begin{bmatrix}
1 & 0 \\
0 & 0 \\
\end{bmatrix} \right) \\
&=  \frac{1}{2} \displaystyle\sup_{B \in \Gamma} \text{tr} \left[
\begin{bmatrix}
1 & 0 \\
0 & 0 \\
\end{bmatrix} B \right] \\
&= \frac{1}{2} \displaystyle\sup_{r_1 \in \left[ \underline{\sigma}^2 , \overline{\sigma}^2 \right] }  r_1 \\
&= \frac{1}{2} \overline{\sigma}^2  , 
\end{align*}
i.e.,
\begin{equation*}
\hat{\mathbb{E}} \left[ X_1^2 \right] =  \overline{\sigma}^2 .
\end{equation*}
A similar calculation gives that
\begin{equation*}
\hat{\mathbb{E}} \left[ X_2^2 \right] =  \overline{\sigma}^2
\end{equation*}
and
\begin{equation*}
-\hat{\mathbb{E}} \left[ - X_1^2 \right] =-\hat{\mathbb{E}} \left[ - X_2^2 \right] =  \underline{\sigma}^2  .
\end{equation*}
By Lemma \ref{basic props known}, $X_1$ and $X_2$ are both $G$-normal random variables and 
\begin{equation*}
X_1 , X_2 \sim \mathcal{N} \left( 0 , \left[ \underline{\sigma}^2 , \overline{\sigma}^2 \right] \right) .
\end{equation*}
Lemma \ref{thm res 3} implies that $X_1$ cannot be independent from $X_2$ and vice versa.
\end{proof}

The backward direction bears a stronger resemblance to the classical case, although with a few unforeseen twists.

\begin{thm}\label{thm res 15}
Suppose that there exists a permutation $\pi \in S_n$ such that for all $1 \leq i \leq n-1$, $X_{\pi\left(i +1 \right)}$ is independent from $\left( X_{\pi\left(1 \right)} , \dots, X_{\pi\left(i \right)}\right)$. Then 
\begin{enumerate}[(i)]
\item for $\underline{\sigma}_i^2$, $\overline{\sigma}_i^2$ such that  $X_i \sim \mathcal{N} \left( 0, \left[ \underline{\sigma}_i^2 , \overline{\sigma}_i^2 \right] \right)$,
\begin{equation*}
\Gamma = \left\{ \text{diag} \left[ r_1 , \dots , r_n \right] : r_i \in \left[ \underline{\sigma}_i^2 , \overline{\sigma}_i^2 \right] \right\} ;
\end{equation*} 
\item for any $1 \leq i \leq n$ such that $0 < \underline{\sigma}_i^2 < \overline{\sigma}_i^2$, 
\begin{equation*}
\alpha \left[ \underline{\sigma}_i^2 , \overline{\sigma}_i^2 \right] \ne \left[ \underline{\sigma}_j^2 , \overline{\sigma}_j^2 \right]
\end{equation*}
for all $j \ne i$ and $\alpha > 0$; and\\
\item for any $i < j$, $X_{\pi \left( i \right)}$ is not independent from $X_{\pi \left(j \right)}$ if either of the following hold: \begin{enumerate}
\item $\underline{\sigma}_{\pi \left(i\right)}^2 < \overline{\sigma}_{\pi \left( i \right)}^2$ and $0 < \overline{\sigma}_{\pi \left(j \right)}^2$, or\\
\item $\underline{\sigma}_{\pi \left(j\right)}^2 < \overline{\sigma}_{\pi \left( j \right)}^2$ and $0 < \overline{\sigma}_{\pi \left(i \right)}^2$.
\end{enumerate} 
\end{enumerate}
\end{thm}

(i) is exactly as expected, but (ii) and (iii) are highly nonintuitive: no remotely similar conditions are present in the classical theory. Observe that while this theorem places substantial restrictions on $\Gamma$ if the coordinates of $X$ satisfy appropriate independence conditions, it does not address the existence of such an $X$. This issue is still unresolved. 

\begin{proof}
We begin with (i). By Lemma \ref{basic props known} we can find $0 \leq \underline{\sigma}_i^2 \leq \overline{\sigma}_i^2$ such that
\begin{equation*}
X_i \sim \mathcal{N} \left( 0, \left[ \underline{\sigma}_i^2 , \overline{\sigma}_i^2 \right] \right)
\end{equation*}
for all $1 \leq i \leq n$. Also, for any $A = \left( a_{ij} \right) \in \mathbb{S} \left(n \right)$,
\begin{equation*}
\frac{1}{2} \hat{\mathbb{E}} \left[ \langle AX , X \rangle \right] = G \left( A \right) = \frac{1}{2} \displaystyle\sup_{B \in \Gamma} \text{tr} \left[ AB \right]   .
\end{equation*}
Now 
\begin{align*}
\frac{1}{2} \hat{\mathbb{E}} \left[ \left\langle AX , X \right\rangle \right] &= \frac{1}{2} \hat{\mathbb{E}} \left[\displaystyle\sum_{i = 1}^n a_{ii}  X_{i}^2 + 2 \displaystyle\sum_{i = 1}^{n-1} \displaystyle\sum_{j = i + 1}^n a_{ij} X_i X_j  \right] \\
&=  \frac{1}{2} \hat{\mathbb{E}} \left[\displaystyle\sum_{i = 1}^n a_{\pi \left(i \right) \pi \left(i \right)}  X_{\pi \left(i \right)}^2 + 2 \displaystyle\sum_{i = 1}^{n-1} \displaystyle\sum_{j = i + 1}^n a_{\pi \left(i \right) \pi \left(j \right)} X_{\pi \left(i \right) } X_{\pi \left(j \right) }  \right]  . 
\end{align*}
 For any $1 \leq i < j \leq n$, 
\begin{equation*}
\hat{\mathbb{E}} \left[ X_{\pi \left(i\right)} X_{\pi \left(j\right)} \right] = \hat{\mathbb{E}} \left[ \hat{\mathbb{E}} \left[  \bar{x} X_{\pi \left(j\right)} \right]_{ \bar{x}=X_{\pi \left(i\right)}} \right] = 0
\end{equation*}
and 
\begin{equation*}
-\hat{\mathbb{E}} \left[- X_{\pi \left(i\right)} X_{\pi \left(j\right)} \right] = -\hat{\mathbb{E}} \left[ \hat{\mathbb{E}} \left[  -\bar{x} X_{\pi \left(j\right)} \right]_{ \bar{x}=X_{\pi \left(i\right)} } \right] = 0 .
\end{equation*}
By repeated application of (\ref{lin mean cert}),
\begin{align*}
\frac{1}{2} \hat{\mathbb{E}} \left[ \left\langle AX , X \right\rangle \right] &= \frac{1}{2} \hat{\mathbb{E}} \left[\displaystyle\sum_{i = 1}^n a_{\pi \left(i \right) \pi \left(i \right)}  X_{\pi \left(i\right)}^2 \right] \\
&= \frac{1}{2} \hat{\mathbb{E}} \left[ \hat{\mathbb{E}} \left[ \cdots \hat{\mathbb{E}} \left[ \displaystyle\sum_{i = 1}^{n-1} a_{\pi \left(i \right) \pi \left(i \right)} \bar{x}_{\pi\left(i\right)}^2 + a_{\pi \left(n \right) \pi \left(n \right)} X_{\pi\left(n \right)}^2 \right]_{\bar{x}_{\pi \left(n-1 \right)}=X_{\pi\left(n -1 \right)} } \cdots \right]_{ \bar{x}_{\pi \left(1 \right)}=X_{\pi\left(1 \right)}} \right] \\
&= \frac{1}{2} \displaystyle\sum_{i = 1}^n  \left( \overline{\sigma}_{\pi \left(i \right)}^2 \left(a_{\pi \left(i \right) \pi \left(i \right)}^+ \right) - \underline{\sigma}_{\pi \left(i \right)}^2 \left(a_{\pi \left(i\right)\pi \left(i \right)}^- \right) \right) \\
&= \frac{1}{2} \displaystyle\sum_{i = 1}^n  \left( \overline{\sigma}_{i}^2 \left(a_{ii}^+ \right) - \underline{\sigma}_{i}^2 \left(a_{ii}^- \right) \right) .
\end{align*}

Hence, we need to find a bounded, closed, convex subset $\Gamma \subset \mathbb{S}^+ \left( n \right)$ such that 
\begin{equation*}
\displaystyle\sum_{i = 1}^n  \left( \overline{\sigma}_{i}^2 \left(a_{ii}^+ \right) - \underline{\sigma}_{i}^2 \left(a_{ii}^- \right) \right)  =  \displaystyle\sup_{B \in \Gamma} \text{tr} \left[ AB \right]
\end{equation*}
for any $A = \left( a_{ij} \right) \in \mathbb{S} \left(n \right)$. One easily verifies that 
\begin{equation*}
\Gamma = \left\{ \text{diag} \left[ r_1 , \dots , r_n \right] : r_i \in \left[ \underline{\sigma}_i^2 , \overline{\sigma}_i^2 \right] \right\} .
\end{equation*}

To prove (ii), suppose that $i \ne j$ and $0 < \underline{\sigma}_i^2 < \overline{\sigma}_i^2$. Let $B_{ij}$ be the $2 \times n$ matrix of 0's and 1's such that 
\begin{equation*}
\begin{bmatrix}
X_i \\
X_j \\ 
\end{bmatrix} = B_{ij} X  .
\end{equation*}
Lemma \ref{basic props known} implies that 
\begin{equation*}
\begin{bmatrix}
X_i \\
X_j \\ 
\end{bmatrix}
\end{equation*}
is a 2-dimensional $G$-normal random vector. By Lemma \ref{thm res 3}, since either $X_j$ is independent from $X_i$ or vice versa, (ii) holds.

(iii) is an immediate consequence of Proposition \ref{hu+li thm}. One might object that the space of test functions in \cite{hu+li} is $C_{b.Lip} \left( \mathbb{R}^n \right)$ instead of $C_{l.Lip} \left( \mathbb{R}^n \right)$, but this issue is addressed in Example 21 of that reference.
\end{proof}

An additional classical result bridging the form of a normal random vector's covariance matrix and the independence of its coordinates is as follows:
\begin{quote}
If $Z$ is an $n$-dimensional normal random vector, then there exists an invertible $n \times n$ matrix $A$ such that the coordinates of $AZ$ are independent.
\end{quote}
The related statement is false for a $G$-normal random vector. There are several possible approaches here. For example, by Lemma \ref{basic props known} and Theorem \ref{thm res 15}, it suffices to construct a $\Gamma$ such that  $A \Gamma A^\top$ contains non-diagonal matrices for all invertible $n \times n$ matrices $A$. Our method uses a more straightforward choice of $\Gamma$ but also incorporates Lemma \ref{thm res 3}.

\begin{thm}\label{invert not indep}
Let
\begin{equation*}
\Gamma = \left\{ \begin{bmatrix}
r_1 & 0\\
0 & r_2 \\
\end{bmatrix} : r_1 , r_2  \in \left[ \underline{\sigma}^2 , \overline{\sigma}^2 \right]  \right\} .
\end{equation*} There is no invertible $2 \times 2$ real matrix $A$ such that the coordinates of $AX$ are independent.
\end{thm}

\begin{proof}
Suppose that for some invertible $2 \times 2$ real matrix $A = \left( a_{ij} \right)$, either $W_2$ is independent from $W_1$ or vice versa, where $\left( W_1 , W_2 \right)^\top = AX$.
By Lemma \ref{basic props known}, 
\begin{equation*}
AX \sim \mathcal{N} \left( 0, A \Gamma A^\top \right) .
\end{equation*}
For all $r_1, r_2  \in \left[ \underline{\sigma}^2 , \overline{\sigma}^2 \right]$, 
\begin{equation*}
A \begin{bmatrix}
r_1 & 0\\
0 & r_2 \\
\end{bmatrix} A^\top = \begin{bmatrix}
r_1 a_{11}^2 + r_2 a_{12}^2  &  r_1 a_{11} a_{21} + r_2 a_{12} a_{22}\\
r_1 a_{11} a_{21} + r_2 a_{12} a_{22} & r_1 a_{21}^2 +  r_2 a_{22}^2 \\
\end{bmatrix}  .
\end{equation*}
From Theorem \ref{thm res 15},
\begin{equation*}
A \Gamma A^\top =  \left\{ \begin{bmatrix}
r_1 & 0\\
0 & r_2 \\
\end{bmatrix} : r_1 \in \left[ \underline{\sigma}_1^2 , \overline{\sigma}_1^2 \right] , r_2 \in \left[ \underline{\sigma}_2^2 , \overline{\sigma}_2^2 \right] \right\}
\end{equation*}
where $W_i \sim \mathcal{N} \left( 0 ,  \left[ \underline{\sigma}_i^2 ,  \overline{\sigma}_i^2 \right] \right)$.

This is only possible if 
\begin{equation*}
a_{11} a_{21} =  a_{12} a_{22} = 0  .
\end{equation*}
Since $A$ is an invertible $2 \times 2$ real matrix, it must be of the form
\begin{equation*}
A =  \begin{bmatrix}
a_{11} & 0 \\
0 & a_{22} \\
\end{bmatrix} \hspace{4mm} \text{for nonzero} \hspace{2mm} a_{11} , a_{22} \in \mathbb{R}
\end{equation*}
or 
\begin{equation*}
A =  \begin{bmatrix}
0 & a_{12} \\
a_{21} & 0 \\
\end{bmatrix} \hspace{4mm} \text{for nonzero} \hspace{2mm} a_{12} , a_{21} \in \mathbb{R} .
\end{equation*}

If the former holds, then 
\begin{equation*}
A \Gamma A^\top =  \left\{ \begin{bmatrix}
r_1 & 0\\
0 & r_2 \\
\end{bmatrix} : r_1 \in \left[ a_{11}^2 \underline{\sigma}^2 , a_{11}^2 \overline{\sigma}^2 \right] , r_2 \in \left[ a_{22}^2 \underline{\sigma}^2 , a_{22}^2 \overline{\sigma}^2 \right] \right\}  .
\end{equation*} 
In this case,
\begin{equation*}
W_1 \sim \mathcal{N} \left( 0 ,  \left[ a_{11}^2 \underline{\sigma}^2 , a_{11}^2 \overline{\sigma}^2 \right] \right) \hspace{4mm} \text{and} \hspace{4mm} W_2  \sim \mathcal{N} \left( 0 , \left[ a_{22}^2 \underline{\sigma}^2 , a_{22}^2 \overline{\sigma}^2 \right]\right)
\end{equation*}
by Lemma \ref{basic props known}, which is impossible by Lemma \ref{thm res 3}.

Similarly, $A$ cannot have the latter form, so the result holds.
\end{proof}

\bibliographystyle{plain}
\bibliography{mybib}

\end{document}